\documentclass[11pt]{article}
\usepackage{graphicx} 
\usepackage{mathrsfs}
\usepackage{indentfirst}
\usepackage{enumerate}
\usepackage{cite}
\usepackage{comment}
\usepackage{color}
\usepackage{amsmath}
\usepackage{cases}
\usepackage{amsthm}
\usepackage{amssymb}
\numberwithin{equation}{section}
\usepackage{mathtools}
\usepackage{stmaryrd}
\theoremstyle{definition}
\newtheorem*{thm*}{Theorem}
\newtheorem{thm}{Theorem}[section]
\newtheorem{prop}[thm]{Proposition}
\newtheorem{lem}[thm]{Lemma}

\newtheorem{remark}[thm]{Remark}
\theoremstyle{plain}
\usepackage{tikz}
\usetikzlibrary{intersections, calc, angles, arrows.meta}
\usetikzlibrary{patterns}

\newcommand{\ep}{\varepsilon}

\renewcommand{\phi}{\varphi}

\newcommand{\N}{\mathbb{N}}

\newcommand{\R}{\mathbb{R}}

\newcommand{\h}{\hspace{2mm}}
\newcommand{\upsi}{\underline{\psi}}
\newcommand{\uPsi}{\underline{\Psi}}

\newcommand{\cE}{\mathcal{E}}

\newcommand{\cF}{\mathcal{F}}

\newcommand{\cL}{\mathcal{L}}

\renewcommand{\hat}{\widehat}
\renewcommand{\Tilde}{\widetilde}

\newcommand{\pl}{\partial}
\newcommand{\norm}[1]{\|#1 \|}

\begin{document}
\title{Grow-up rates of inhomogeneous semilinear heat equations in the critical dimension}
\author{Kenta Kumagai}
\maketitle

\begin{abstract}
This paper concerns the grow-up rate of solutions to a semilinear heat equation in the unit ball with the exponential nonlinearity and an inhomogeneous term $f$. When $f=0$, it is known that the large-time behavior of a solution changes at the critical dimension $N=10$, and the grow-up phenomenon occurs for the case $N\ge 10$. For the inhomogeneous case with $N\ge 10$, the present author and a coauthor showed in \cite{KO26} that the grow-up phenomenon disappears once $f$ exceeds a threshold. 

In this paper, we provide a quantitative characterization of the disappearance of the grow-up phenomenon by obtaining the grow-up rate in the critical dimension. Our result shows that a qualitatively different type of grow-up behavior occurs in the critical dimension compared with the case $N\ge 11$ studied in \cite{KO26}. The difference is caused by a change in the outer behavior of the solution. Even in the case $f=0$, our result is new in that it provides a rigorous justification of the formal computation by Galaktionov and King \cite{GK}.
\end{abstract}
\noindent Addresses:

\smallskip
\noindent
K.~K.: Graduate School of Mathematical Sciences, The University of Tokyo,\\
3-8-1 Komaba, Meguro-ku, Tokyo 153-8914, Japan.

\smallskip
\noindent
E-mail: {\tt kumagai-kenta@g.ecc.u-tokyo.ac.jp}\\


\vspace{3pt}

\vspace{3pt}


\section{Introduction}
We consider the following semilinear heat equation
\begin{align}\label{eq-intro-1}
\begin{cases}
\partial_t u-\Delta u=\lambda e^u-f(r),
&\text{$x\in B_{1}$, $t>0$,} \quad u=0, \quad \text{$x\in \pl B_{1}$, $t>0$},
\\
u(x,0)=u_0(x) \in C^{0}(\overline{B_1}) &\text{$x\in B_1$},
\end{cases}
\end{align}
where $r:=|x|$, $\lambda>0$, $B_1\subset\R^{N}$ is the unit ball with $N\geq 3$ and 
\begin{equation}
\label{katei-f}
0\le f(r)\in \mathrm{Lip}[0,1].
\end{equation}
We assume throughout this paper that
$u_{0}\ge \phi_0$ in $B_1$, where $\phi_0$ is the solution of
\begin{equation*}
-\Delta \phi_0 = -f(r) \quad \text{in $B_1$,}\quad \phi_0=0 \quad \text{on $\partial B_1$}.
\end{equation*}
The assumption $u_{0}\ge \phi_0$ is natural since every stationary solution $u$ satisfies $u \ge\phi_0$ in $B_1$ by the maximum principle.

The dynamics of \eqref{eq-intro-1} exhibit a qualitative change at the critical parameter $\lambda^{*}\in (0,\infty)$. More precisely, the following properties are known (see, \cite{Fu1969, BCMR, PV1995,JL}). If $\lambda<\lambda^{*}$, \eqref{eq-intro-1} has a unique stable stationary solution $u_{\lambda}\in C^2(\overline{B_1})$ in the sense that
\begin{equation*}
\int_{B_1}|\nabla \xi|^2-\lambda e^{u_{\lambda}} \xi^2\,dx\ge 0 \quad \text{for any $\xi\in C^{1}_{c}(B_1)$.}
\end{equation*}
Moreover, for any $u_0\le u_{\lambda}$ in $B_1$, \eqref{eq-intro-1} admits a unique global-in-time solution $u$ such that $u\to u_{\lambda}$ in $C^2 (B_1)$ as $t\to\infty$. On the other hand, if $\lambda>\lambda^{*}$, \eqref{eq-intro-1} admits no (even weak) stationary solution. Moreover, for any $u_{0}\in C^{0}(\overline{B_1})$, \eqref{eq-intro-1} admits a unique local-in-time solution which blows up in finite time. Finally, in the critical case $\lambda=\lambda^{*}$, problem \eqref{eq-intro-1} admits a unique stationary solution $u^{*}\in H^{1}_{0,\mathrm{rad}}(B_1)\cap C^{2}_{\mathrm{loc}}(\overline{B_{1}}\setminus \{0\})$. In addition, for any $u_0\le u^{*}$ in $B_1$, \eqref{eq-intro-1} admits a unique global-in-time solution $u$ such that $u\to u^{*}$ in $L^2 (B_1)\cap C^{2}_{\mathrm{loc}}(\overline{B_{1}}\setminus \{0\})$ as $t\to\infty$.

In the critical case $\lambda=\lambda^{*}$, the large-time behavior of solutions depends on the dimension $N$. Indeed, for the case $f=0$, it is known in \cite{JL,PV1995, Fu1969} that $u^{*}\in C^2(\overline{B_1})$ and hence $u\to u^{*}$ in $C^{2}(B_1)$ as $t\to\infty$ when $N\le 9$. On the other hand, we have $u^{*}\not\in L^{\infty}(B_1)$ and hence $\lVert u\rVert_{L^{\infty}}\to \infty$ as $t\to \infty$ when $N\ge 10$ (i.e., the grow-up phenomenon occurs in this case).
Motivated by this result, the present author and a coauthor \cite{KO26} studied the effect of $f$ on the asymptotic behavior of $u$ for the critical case $\lambda=\lambda^{*}$, and showed that the grow-up phenomenon disappears when $f$ exceeds the threshold $f_H$, where
\begin{equation}
\label{defH}
H:=\text{the first eigenvalue of }
\begin{cases}
\text{$-\Delta_D$ in $B_{1}^{2}\subset \mathbb{R}^2$}& \text{if $N=10$,} \\
\text{$-\Delta_D-\frac{2N-4}{|x|^2}$ in $B_{1}^{N}\subset \mathbb{R}^N$} & \text{if $N\ge 11$}\\
\end{cases}
\end{equation}
and $f_h$ is defined as
\begin{equation}
\label{deffh}
f_h= h\left(8(N-2)a^{2}(r)+2(N-2)a(r)+1\right) \quad \text{with} \quad a(r)=\frac{1}{2N-4+hr^2} \quad
\end{equation}
for any parameter $h\ge 0$.
More precisely, the following are proved in \cite{KO26}:
\begin{prop}
\label{intro-prop-1}
Let $N\ge 3$ and $\lambda=\lambda^{*}$. Assume that \eqref{katei-f}, $u_0\in C^{0}(\overline{B_1})$ and $\phi_0\le u_0\le u^{*}$ in $B_1$. 
Then, \eqref{eq-intro-1} admits a unique global solution $u$ such that $u\to u^*$ in $L^2(B_1)\cap C^2_{\mathrm{loc}}(\overline{B_1}\setminus \{0\})$ as $t\to \infty$. Furthermore, $u^{*}\in H^{1}_{0,\mathrm{rad}}(B_1)$ and $u$ satisfy
\begin{enumerate}
\item[(i)] $u^{*}\in C^{2}(\overline{B_1})$ and thus
$u\to u^{*}$ in $C^{2}(B_1)$ as $t\to\infty$ when $N\le 9$;
\item[(ii)] $u^{*}\not \in L^{\infty}(B_1)$ and
thus $\lVert u\rVert_{L^{\infty}(B_1)}\to\infty$ as $t\to\infty$ when $N\ge 10$ and $f\le f_H$ in $B_1$;
\item[(iii)] $u^{*}\in C^{2}(\overline{B_1})$ and
thus $u\to u^{*}$ in $C^{2}(B_1)$ as $t\to\infty$ when $N\ge 10$, $f\ge f_H$ in $B_1$
and $f\not \equiv f_H$,
\end{enumerate}
where $H$ and $f_h$ are those in \eqref{defH} and \eqref{deffh}, respectively,
\end{prop}

The change in the large-time behavior of $u$ shown in Proposition \ref{intro-prop-1} is caused by a change in the stability of singular radial stationary solutions\footnote{Here, we say that $(\lambda_*, U_{*}(r))$ is a singular radial stationary solution if $U_{*}(r)\in C^2_{\mathrm{loc}}(0,1]$ satisfies the equation in $B_{1}\setminus \{0\}$ for $\lambda=\lambda_*$
together with the Dirichlet boundary condition and $U_{*}(r)\to \infty$ as $r\to 0$.}. Indeed, 
it is shown in \cite{BV, KO26} that
\eqref{eq-intro-1} has a unique singular radial stationary solution $(\lambda_{*}, U_{*})$. Moreover, the singular solution is stable if and only if $u^{*}\not \in L^{\infty}$. The results show that the grow-up phenomenon occurs precisely in the stable case. Note that this relationship is well-known in \cite{BV,MN20,MN23} for the case $f=0$. 

For the parameterized families of inhomogeneous terms $f=f_h$, a family of singular radial stationary solutions is explicitly constructed in \cite{KO26} as follows:
\begin{equation*}
(\lambda_h, U_{h})=(2N-4+h, -2\log r-\log a(r)-\log (2N-4+h)).
\end{equation*}
The explicit form of the singular solutions enables us to study the stability of $U_{h}$. As a result, the authors \cite{KO26} showed by using a separation result that the stability/instability of $U_{*}$ changes at $f=f_H$ for the case $N\ge 10$, which explains the disappearance of the grow-up phenomenon qualitatively.

We aim to provide a quantitative characterization of the disappearance of the grow-up phenomenon by obtaining the grow-up rates. It is worth noting that the analysis of grow-up rates has been extensively studied in the literature (see, e.g., \cite{CZ2022, Mizo06, GK, PY03, FKWY06, FKWY07-ADE, FKWY07,DGLV1998,FWY} for related problems). For the problem \eqref{eq-intro-1}, the first eigenvalue $\mu_1$ of the following linearized operator 
\begin{equation*}
\mathcal{L}=-\Delta-\lambda^{*}e^{u^{*}}
\end{equation*}
plays a key role in the characterization of the grow-up rate. In fact, the present author and a coauthor \cite{KO26}
determined the grow-up rate in terms of $\mu_1$ for $N\ge 11$.
\begin{thm}[see \cite{KO26}]
\label{intro-thm-1}
Assume that $N\ge 11$ and $0\le f\le f_H$ in $B_1$ in addition to the hypotheses of Proposition \ref{intro-prop-1}. 
Then, the grow-up rate of the solution to \eqref{eq-intro-1} is characterized as follows:
\begin{enumerate}
\item[(i)] $\lVert u\rVert_{L^{\infty}(B_1)}=\frac{2\mu_1}{\gamma}t+O(1)$ as $t\to\infty$ if $f\not \equiv f_{H}$;
\item[(ii)] $\lVert u\rVert_{L^{\infty}(B_1)}=\frac{2}{\gamma}\log t+O(1)$ as $t\to\infty$ if $f\equiv f_H$ and $N\ge 12$;
\item[(iii)] $\lVert u\rVert_{L^{\infty}(B_1)}=\frac{2}{\gamma}(\log t+\log \log t)+O(1)$ as $t\to\infty$ if $f\equiv f_H$ and $N=11$,
\end{enumerate}
where $\mu_1$ is the first eigenvalue of $\mathcal{L}=-\Delta - \lambda^{*}e^{u^{*}}$ and $\gamma$ is defined by
\begin{equation}
\label{defgam}
\gamma:=\frac{1}{2}\left(N-2-\sqrt{(N-2)(N-10)}\right) \h \text{(in particular, $\gamma=4$ for
$N=10$)}.
\end{equation}
\end{thm}
\begin{remark}
\rm{The operator $\cL$ has discrete spectrum, and its eigenfunctions form an orthonormal basis of $L^2(B_{1})$ when $N\ge 10$. The first eigenvalue $\mu_1$ satisfies $\mu_1(f_1)\ge \mu_1(f_2)>0$ if $f_1\le f_2\le f_H$ and $f_2\not\equiv f_H$ in $B_1$. Moreover, 
the map $f\mapsto \mu_1 $ is continuous and $\mu_1\downarrow 0$ as $f\uparrow f_H$. For the special case $f=f_h$, the first eigenvalue is given by $\mu_1=H-h$. The precise statements are summarized in Proposition \ref{evprop}. }
\end{remark}
For the case $f=0$ and $N\ge 11$, Dold, Galaktionov, Lacey and V\'azquez \cite{DGLV1998} obtained the same grow-up rate. Theorem \ref{intro-thm-1} includes the result in \cite{DGLV1998} as a special case. In addition, Theorem \ref{intro-thm-1} gives a quantitative understanding of the disappearance of the grow-up phenomenon. In fact, as $f\uparrow f_H$, the coefficient of the leading-order term tends to $0$ because $\mu_1\downarrow 0$. Furthermore, at the threshold case $f=f_H$, the grow-up rate becomes logarithmic. We also observe that, an additional log-log type correction term emerges for the dimension $N=11$.

The main result in this paper shows that a different type of grow-up behavior appears in the critical dimension $N=10$. 
\begin{thm}
\label{intro-thm-2}
Assume that $N=10$ and $0\le f\le f_H$ in $B_1$ in addition to the hypotheses of Proposition \ref{intro-prop-1}.
Then, the grow-up rate of the solution to \eqref{eq-intro-1} is characterized as follows:
\begin{align*}
\text{(i) }\lVert u\rVert_{L^{\infty}(B_1)}&=\frac{\mu_1}{2} t +\frac{1}{2}\left(1+\frac{1}{\mu_1 b}\right)\log t+O(1)\text{ as $t\to\infty$ if $f\not \equiv f_{H}$ in $B_{1}$;}\\
\text{(ii) }
\lVert u\rVert_{L^{\infty}(B_1)}&=(\frac{t}{2b})^{1/2}+O(\log t) \h \text{ as $t\to\infty$ if $f\equiv f_H$ in $B_{1}$,}
\end{align*}
where $\mu_1$ is that in Theorem \ref{intro-thm-1} and $b>0$ is a constant depending only on $f$ and defined in \eqref{defb}.
\end{thm}
\begin{remark}
\rm{We note that $f\mapsto b$ is continuous. In particular, we have $c<b<C$ for any $0\le f\le f_H$ with some $0<c<C$ independent of $f$. For the special case
$f=f_h$, $b$ is explicitly given by $b=J_{1}^{2}(\sqrt{H})/8$, where $J_1$ is the Bessel function of the first kind of order $1$. The precise statements are summarized in Lemma \ref{blem}.}
\end{remark}
For the case $N=10$ and $f=0$, Galaktionov and King \cite{GK} derived the same grow-up rate by a formal computation, while the rigorous justification remains open due to the analytical difficulties.
Theorem \ref{intro-thm-2} gives a rigorous justification to their formal computation for the case $f=0$. Moreover, Theorem \ref{intro-thm-2} shows that in the critical dimension $N=10$, the transition of grow-up rates near the threshold $f=f_H$ is essentially different from that in dimensions $N\ge 11$. Indeed, when $f\not \equiv f_H$ and $f\le f_H$, a logarithmic correction term appears in contrast to the case $N\ge 11$.
Furthermore, this correction term becomes larger as $f$ approaches $f_H$, and turns into the leading-order contribution at the threshold $f=f_H$. Consequently, the algebraic grow-up rate $t^{1/2}$ emerges at the threshold, which is intermediate between linear and logarithmic rates.

We now explain the mechanisms behind the different grow-up rates obtained in Theorems \ref{intro-thm-1} and \ref{intro-thm-2}. The differences in these grow-up rates are reflected in the behavior of the solution in the outer region.
In the case $N\ge 11$, it is shown in \cite{KO26} that the outer behavior of $\Phi:=u^{*}-u$ is 
governed by its projection $Q(t)\phi_1$
onto the first eigenspace of $\mathcal{L}$, where $\phi_1>0$ is the first eigenfunction satisfying $\norm{\phi_1}_{L^2(B_1)}=1$. The decay rate of $Q(t)$ changes depending on $f$ and it yields the different grow up rates stated in Theorem \ref{intro-thm-1}. We refer to \cite[Subsection 4.1]{KO26} for details.

On the other hand, the situation is qualitatively different in the case $N=10$. In fact, we show the following
\begin{thm}
\label{intro-thm-3}
Assume the hypotheses of Theorem \ref{intro-thm-2}. Let $\Phi:=u^{*}-u$ and $\phi_1>0$ be the principal eigenfunction of $\cL$ with $\norm{\phi_1}_{L^2(B_1)}=1$. 
Then, there exist $M,C,T>0$ depending only on $u_0$ and $f$ such that 
\begin{equation*}
\frac{1}{C}e^{-\mu_1 t}\Bigl(\zeta_2(t)\phi_1+\nu |\zeta_2'(t)|r^{-4}\log r \Bigr)- Me^{-\mu_1 t}|\zeta_2'(t)| \phi_1 \le \Phi \quad \text{in $B_1\times (T,\infty)$}
\end{equation*}
and
\begin{equation*}
   \Phi \le C e^{-\mu_1 t}\Bigl(\zeta_1(t)\phi_1-\nu |\zeta_1'(t)| r^{-4}\Gamma(r,t) \Bigr) \quad \text{in $B_1 \times (T,\infty)$.}
\end{equation*}
Here, $\Gamma=\Gamma(r,t)$ is a smooth regularization of $-\log r$ near $r=0$ defined in \eqref{defGam}. The functions $\zeta_1$ and $\zeta_2$ are defined in Lemma \ref{zetalem}, and the constant $\nu>0$ is defined in \eqref{defb}.
\end{thm}
\begin{remark}
\rm{The functions $\zeta_1$ and $\zeta_2$ have the following asymptotics 
\begin{align*}
\zeta_1\simeq\zeta_2\simeq t^{-\frac{1}{\mu_1 b}} & \quad \quad \text{if $f\not \equiv f_H$ in $B_{1}$,}\\
\zeta_1\simeq e^{-(2t/b)^{1/2}}t^{1/2}, \quad \zeta_2\simeq  e^{-(2t/b)^{1/2}}t^{-5} &\quad \quad \text{if $f\equiv f_H$ in $B_1$}
\end{align*}
as $t\to\infty$, where $b>0$ is defined in \eqref{defb}.}
\end{remark}
Theorem \ref{intro-thm-3} shows that, in contrast to the case $N\ge 11$, the leading-order expansion of $\Phi $
contains not only the first-eigenfunction component but also an additional logarithmic-type term.
The difference in the asymptotic profile yields the grow-up rate obtained in Theorem \ref{intro-thm-2}, which is essentially different from the case $N \ge 11$. The proof of Theorem \ref{intro-thm-3} is based on constructing super/sub-solutions that 
reflect the relevant asymptotic structure of the solution in the outer region. We emphasize that the construction is considerably more delicate than in the case $N\ge 11$, due to the presence of logarithmic-type term in the leading-order behavior of $\Phi$. The novelty of this paper lies in rigorously clarifying the distinctive grow-up mechanism in the critical dimension $N=10$ through the construction of super/sub-solutions.

The paper is organized as follows. In Section 2, we recall known results about a singular radial stationary solution and properties of $\mathcal{L}$. In addition, we define $b$ and $\nu$ and study the properties of $b$. In Section 3, we show Theorems \ref{intro-thm-2} and \ref{intro-thm-3} by constructing the super/sub-solutions.

\section{Preliminaries}
In this section, we quote some results in \cite{KO26} and introduce some lemmata, which is used in order to obtain the grow-up rate.
\subsection{Singular solution for stationary problem}
We first focus on the radial singular stationary solution of \eqref{eq-intro-1}. Here, we say that $(\lambda_{*}, U_{*}(r))$ is a radial singular stationary solution for \eqref{eq-intro-1} if $U_{*}\in C^{2}_{\mathrm{loc}}(0,1]$ satisfies \eqref{eq-intro-1} in $(0,1]$ with $\lambda=\lambda^{*}$ such that $U_{*}(r)\to \infty$ as $r\to 0$. By combining Propositions 1.1, 1.2, Remark 2.13 and Theorem 1.3 
obtained in \cite{KO26}, we obtain the following

\begin{prop}
\label{prop-2-1}
Assume that \eqref{katei-f}. Then, \eqref{eq-intro-1} admits the unique singular radial stationary solution $(\lambda_{*}, U_{*})$. Moreover, $U_{*}\in C^2_{\mathrm{loc}}(0,1]\cap H^{1}_{0}(B_1)$ and $U_{*}$ is
\begin{enumerate}
    \item[(i)] unstable if $N\le 9$;
    \item[(ii)] stable and $u^{*}=U_{*}$ if $N\ge 10$ and $f\le f_H$ in $B_{1}$;
    \item[(iii)] unstable if $N\ge 10$, $f\ge f_H$ in $B_{1}$ and $f\not \equiv f_H$.
\end{enumerate}
In addition, for the case $N\ge 10$, we have 
\begin{equation*}
    K(r):= \lambda_{*}e^{U_{*}}-\frac{2N-4}{r^2}\in C^{0}[0,1].
\end{equation*}
\end{prop}
In the following, we only consider the case $N\ge 10$ and $f\le f_H$ in $B_1$. Thanks to Proposition \ref{prop-2-1}, we have $\cL:=-\Delta -\lambda^{*}e^{u^{*}}=-\Delta-\lambda_* e^{U_{*}}$. By studying the global behavior of a singular solution via the separation results, the authors \cite{KO26} obtained the following properties of $\cL$. 
\begin{prop}[see \cite{KO26}]
\label{evprop}
Assume that \eqref{katei-f}, $N\ge 10$ and $f\le f_H$ in $B_1$. Then, we obtain the following.
\begin{enumerate}
    \item[(i)] The operator $\cL$ has a countable sequence of discrete eigenvalues, and the corresponding eigenfunctions form an orthonormal basis of $L^2(B_1)$. Moreover, the first eigenvalue is simple.
    \item[(ii)]  The restriction of $\cL$ to $L^{2}_{\mathrm{rad}}(B_1)$ has an orthonormal basis of eigenfunctions $\{\phi_k\}_{k\in \N}$ with corresponding eigenvalues $\{\mu_k\}_{k\in \N}$. Each eigenspace is one-dimensional. Thus,
    after the normalization $\phi_k(r)>0$ near $r=0$, $\phi_k$ is uniquely determined. Moreover, $\mu_1$ and $\phi_1$
    coincide with the principal eigenvalue and a corresponding eigenfunction of $\cL$ in $L^2 (B_1)$, respectively. Furthermore, $\phi_1>0$ in $B_1$.
    \item[(iii)] Let $m=N+2\gamma$, where $\gamma$ is that in \eqref{defgam}. We define
    \begin{equation*}
   \psi_k(r):=(\frac{N\omega_N}{m\omega_m})^{1/2}r^{\gamma}\phi_k,
    \end{equation*}
    where $\omega_{N}$ is the volume of the $N$-dimensional unit ball.
    Then, $\psi=\psi_k$
    is an orthonormal basis in $L^{2}_{\mathrm{rad}}(B_{1}^m)$ 
    for the following eigenvalue problems
    \begin{equation*}
    -\Delta\psi - K(r) \psi=\mu\psi \quad \text{in $B_{1}^m$} \quad \psi=0 \quad \text{on $\partial B_{1}^m $}
    \end{equation*}
    with $\mu=\mu_k$.
 Moreover, $\psi_k\in C^2(\overline{B_{1}^m})$ and $\psi_1>0$, $\psi_1'\le 0$ in $B_{1}^m $.
\item[(iv)]  In the case $f=f_h$, we have $\mu_1=H-h$. Moreover, if $f_1\le f_2\le f_H$ and $f_2\not\equiv f_H$ in $B_1$, then $\mu_1(f_1)\ge \mu_1(f_2)> 0$. Finally, $\mu_1\downarrow 0$ as $f\uparrow f_H$ uniformly in $B_1$.  
\item[(v)] The following maps are continuous:
\[\mathrm{Lip}[0,1]\cap \{f\le f_H \} \ni f\mapsto \mu_1 \in [0,\infty), \h \mathrm{Lip}[0,1]\cap \{f\le f_H \} \ni f\mapsto  \psi_1(0) \in (0,\infty).
\]
\end{enumerate}
\end{prop}
\begin{remark}
\rm{Proposition \ref{evprop} is shown by combining 
Proposition 2.12, 
Lemmata 2.9 and 2.21 in \cite{KO26}. We refer to \cite[Appendix B]{KO26} for the definitions of $m$-dimensions.}
\end{remark}

\subsection{Definitions and properties of $b$ and $\nu$}
\label{defsec}
Throughout this subsection, we assume that \eqref{katei-f},
$N\ge 10$ and $f\le f_H$ in $B_1$. We define
\begin{equation}
\label{defb}
\nu= \frac{1}{2\pi\psi_1(0)}\sqrt{\frac{\omega_2}{5\omega_{10}}}, \quad \nu_0=\frac{1}{2\pi \psi_1(0)},\quad b=\frac{\nu_0}{4\psi_1(0)}=\frac{1}{8\pi \psi_{1}(0)^2}.
\end{equation}
Then, we obtain the following.
\begin{lem}
\label{blem}
\begin{enumerate}
    \item[(i)] The following identity is satisfied:
    \begin{equation}
    \label{pnu0}
        -\int_{B_{1}^{2}}(K+\mu_1)\log r\psi_1\,dx=\frac{1}{\nu_0}.
    \end{equation}
    \item[(ii)] The following map is continuous:
    \[\mathrm{Lip}[0,1]\cap \{f\le f_H \} \ni f\mapsto b \in (0,\infty).\]
In particular, there exist positive constants $c<C$ independent of $f$ such that $c<b<C$ for any $f$ satisfying \eqref{katei-f} and $f\le f_H$ in $B_1$.
    \item[(iii)] When $f=f_h$ with some $h\ge 0$, then $\mu_1=H-h$ and $b=J_{1}^{2}(\sqrt{H})/8$, where $J_1$ is the Bessel function of the first kind of order $1$. 
\end{enumerate}
\end{lem}
\begin{proof}
Since $-\Delta (-\log r)=2\pi \delta_0$ in $B_{1}^2$, it follows from the Green's identity that
\begin{equation*}
-\int_{B_{1}^{2}}(K+\mu_1)\log r\psi_1\,dx= \int_{B_{1}^{2}}\log r \Delta \psi_1\,dx= \int_{B_{1}^{2}}\Delta (\log r)\psi_1\,dx=2\pi\psi_1(0)=\frac{1}{\nu_0}.
\end{equation*}
Hence, \eqref{pnu0} holds. Moreover, the assertion  (ii) follows from Proposition \ref{evprop} and the Arzel\`a-Ascoli theorem.

We now prove the assertion (iii). For the case $f=f_h$, we have $K=h$ and $\mu_1=H-h$ by Proposition \ref{evprop}. Moreover, by the definition and fundamental properties of Bessel functions, we obtain 
\begin{equation*}
  \psi_1 (r) = \frac{J_{0}(\sqrt{H}r)}{\sqrt{\pi} J_{1}(\sqrt{H})}
\end{equation*}
because
\[
\lVert J_0(\sqrt{H}r)\rVert_{L^2}^2=2\pi\int_{0}^{1}J_0(\sqrt{H}r)^2 r\,dr=\pi J_{1}(\sqrt{H})^2,
\]
where $J_{\rho}$ is the Bessel function of the first kind of order $\rho$. 
Since $J_0(0)=1$, we have $b=J_{1}^{2}(\sqrt{H})/8$.
\end{proof}
\section{Grow-up rate}
In this section, we show Theorems \ref{intro-thm-2} and \ref{intro-thm-3}. 
Let $\lambda=\lambda^{*}$. We recall that $\phi_0$ and $u^{*}$ are radial. Hence, for each $u_0\in C^{0}(\overline{B_1})$ satisfying $\phi_0\le u_0\le u^{*}$ in $B_{1}$,
there exist radial functions $u_{0}^{(1)}$, $u_{0}^{(2)}\in C^{0}(\overline{B_1})$ such that $\phi_0\le u_{0}^{(1)}\le u_0\le u_0^{(2)}\le u^{*}$ in $B_1$. Then, by Proposition \ref{intro-prop-1} and a comparison principle,  the global solution $u$ satisfies $\phi_0\le u^{(1)}\le u\le u^{(2)}\le u^{*}$ in $B_1\times (0,\infty)$, where $u^{(i)}$ are the radial global solutions of \eqref{eq-intro-1} satisfying $u^{(i)}(r,0)=u_{0}^{(i)}(r)$ in $B_1$ for 
$i=1,2$. Therefore, it is sufficient to prove Theorems \ref{intro-thm-2} and \ref{intro-thm-3} under the assumption that $u_{0}$ is radially symmetric.

Throughout this section, we assume without loss of generality that \eqref{katei-f}, $N=10$, $0\le f\le f_H$ in $B_1$, $\lambda=\lambda^{*}$, $\phi_0\le u_0\le u^{*}$ in $B_1$, 
and $u_{0}\in C^{0}(\overline{B_1})$ is radially symmetric. 
Then, Proposition \ref{eq-intro-1} shows that $u$ is radially symmetric and $u\to u^{*}$ in $C^{2}_{\mathrm{loc}}(\overline{B_1}\setminus \{0\})$ as $t\to\infty$. Moreover, by a comparison principle, the strong maximum principle and the Hopf lemma, we deduce that $\Phi:=u^{*}-u$ satisfies 
\begin{equation}
\label{sita}
   c\psi_1<\Phi(r, 1)\quad \text{in $B_1$} \quad \text{and}\quad \Phi\ge 0 \quad \text{in $B_{1}\times (0,\infty)$} 
\end{equation}
with some $c>0$.
In addition, we remark that $0\le \Phi\le u^{*}-\phi_0$ in $B_1\times (0,\infty)$ by a comparison principle. Therefore, it follows from Proposition \ref{prop-2-1} that
\begin{equation}
\label{ue}
    r^4 \Phi(r,t)\to 0 \quad \text{in $C^{2}_{\mathrm{loc}}(\overline{B_1}\setminus \{0\})\cap C^{0}(B_1)$} \quad \text{as $t\to\infty$.}
\end{equation}

\subsection{Proof of Theorem \ref{intro-thm-3}}
In this subsection, we prove Theorem \ref{intro-thm-3} by constructing super/sub-solutions. For $\tau\in \R$, we apply the following transformation $\Psi(r,t,\tau)=e^{\mu_1 t}r^{4}\Phi(r, t+\tau)$. Then, $\Psi(r,t,\tau)$ satisfies
\begin{equation*}
(\partial_t+ \cE)\Psi=\lambda^{*}e^{u^{*}}e^{\mu_{1}t}r^{4}\cF(r^{-4}e^{-\mu_1 t}\Psi) \quad \text{in $B_{1}^2 \times (-\tau ,\infty)$,}
\end{equation*}
where 
\begin{equation*}
    \cE=-\Delta-K(r)-\mu_1 \quad \text{and} \quad \cF= e^{-u_+ }+u_+ -1.
\end{equation*}
We deduce from Proposition \ref{evprop} that
$\cE$ is degenerate. Moreover, we obtain 
\begin{equation}
0\le \cF'(u)\le \min\{u_{+},1\}, \hspace{2mm}
\label{pF}
\begin{cases}
\frac{u^2}{2e^2} \le \cF(u)\le \frac{u^2}{2}, &0\le u\le 2,\\
\frac{u}{2}\le \cF(u)\le u &u\ge 2.
\end{cases}
\end{equation}
Our goal is to find a super-solution with the following form
\begin{equation*}
\Psi_1=\zeta_1(t)\psi_1+(\nu_0+o(1))\zeta_{1}'(t)\Gamma(r)+O(|\zeta_1'|\psi_1) \quad \text{as $r\to\infty$}
\end{equation*}
and a sub-solution with the following form
\begin{equation*}
\Psi_2=\zeta_2(t)\psi_1-(\nu_0+o(1))\zeta_2'(t)\log r+O(|\zeta_2'|\psi_1) \quad \text{as $r\to\infty$},
\end{equation*}
where $\zeta_1$ and $\zeta_2$ are suitable functions and $\Gamma$ is a smooth regularization of $-\log r$.

First, we aim to define functions $\zeta_{1}(t)$ and $\zeta_2(t)$. In order to define $\zeta_1$ and $\zeta_2$, we need several preliminaries. We start by taking $\eta$ so that
\begin{equation}
\label{defeta}
\text{$\eta\in C^2[0,\infty)$, \h $\eta'\le 0$ in $[0,\infty)$, \h$\eta=1$ in $[0,1]$, \h and \h $\eta=0$ in $[2,\infty)$}
\end{equation}
and we denote by
\begin{equation*}
J(r,s):= -\log r+\eta(r/s)\log(r/s) \quad 0<r,s<1.
\end{equation*}
Then, there exists $M_{1}>0$ depending only on $\eta$ and $\nu_0$ such that
\begin{equation}
\label{defm1}
-\log \max\{r,s\}\le J\le -\log r+\log 2,\h
|\Delta_r J|\le \frac{M_{1}}{2\nu_0} s^{-2}\chi_{s<r<2s} \h\text{and}\h |J_{s}|\le \frac{M_1}{2\nu_0} s^{-1}\chi_{r<2s}
\end{equation}
for any $r,s<1$. Moreover, for any $g\in L^3_{\mathrm{rad}}(B_{1}^2)$ with $g \perp \psi_1$ in $L^2(B_{1}^2)$,
we define $\xi_0\in H^{1}_{0,\mathrm{rad}}(B_{1}^2)$ as the unique solution to the equation
\begin{equation}
\label{xi_0}
\cE \xi_0= g\quad \text{in $B_{1}^2$} \quad \xi_0=0 \quad \text{on $\partial B_{1}^{2}$,} \quad \text{$\xi_{0}\perp \psi_1$ in $L^{2}(B_{1}^2)$}
\end{equation}
and define 
\begin{equation*}
\hat q
:=
\sup \left\{
\lVert\frac{\xi_0}{\psi_1}\rVert_{L^{\infty}(B_{1}^2)}
\; ; \;
\begin{array}{l}
\xi_0 \text{ is a solution of \eqref{xi_0} with some }
g\in L^3_{\mathrm{rad}}(B_1^2) \\ \text{satisfying $g \perp \psi_1$ in  $L^2(B_1^2)$ and $\|g\|_{L^3}\le 1$}
\end{array}
\right\}.
\end{equation*}
Then, we show that $\hat{q}$ is well-defined. More precisely, we obtain the following
\begin{lem}
\label{well-def-lem}
For any $g\in L^3(B_{1}^{2})$ with $g\perp \psi_1$ in $L^2(B_{1}^2)$, the solution $\xi_0$ satisfies
$\xi_0\in C^{1}_{0}(\overline{B_1})$. Moreover,
there exists $C>0$ depending only on $K$, $\mu_1$, $\mu_2$ and $\psi_1$ such that $\lVert \xi_0\rVert_{C^{1}(B_{1}^2)}<C\lVert g\rVert_{L^3}$ and
thus $|\xi_0|\le C\lVert g\rVert_{L^3}\psi_1$ in $B_{1}^2$. 
\end{lem}
\begin{proof}
By the energy estimate and the Poincar\'e inequality, we obtain
\begin{align*}
(\mu_2-\mu_1)\lVert \xi_0&\rVert_{L^2(B_{1}^{2})}^2\le \langle\cE \xi_0,\xi_0\rangle\le \lVert g\rVert_{L^{2}(B_{1}^2)}\lVert \xi_0\rVert_{L^2(B_{1}^2)}\le \pi^{1/3}\lVert g\rVert_{L^3(B_{1}^2)}\lVert\xi_0\rVert_{L^2(B_{1}^2)};\\
\lVert \nabla \xi_{0}\rVert_{L^2(B_{1}^2)}^2&=\langle \cE \xi_0,\xi_0\rangle+ \lVert (\mu_1+K)\xi_0^2\rVert_{L^1(B_1^2)}.
\end{align*}
Hence, we have $\lVert \xi_{0}\rVert_{H^{1}_{0}(B_{1}^2)}<C \lVert g\rVert_{L^3}$ and thus $\lVert \xi_{0}\rVert_{L^{3}(B_1^2)}\le C \lVert g\rVert_{L^3}$
where $C>0$ is a constant depending only on $K$, $\mu_1$ and $\mu_2$. Therefore, by the standard regularity theory, we have
\begin{align*}
\lVert \xi_0 \rVert_{C^1(B_{1}^2)}\le C\lVert \xi_0\rVert_{W^{2,3}(B_{1}^{2})}
&\le C(\lVert \xi_0\rVert_{L^3(B_{1}^{2})}+\lVert (K+\mu_1)\xi_{0}\rVert_{L^3(B_{1}^{2})}+\lVert g\rVert_{L^3(B_{1}^2)})\\
&\le C \lVert g\rVert_{L^3},
\end{align*}
where $C$ is depending only on $K$, $\mu_1$ and $\mu_2$. Hence, the result follows from the Hopf lemma.
\end{proof}
We define
\begin{equation}
\label{defm2}
M_2=4\hat{q}(\nu_0 \lVert(K+\mu_1)(-\log r+\log 2)\rVert_{L^3(B_{1}^2)}+\lVert \psi_1\rVert_{L^3(B_{1}^2)})
\end{equation}
and fix $l>0$ such that 
\begin{equation}
\label{asL}
L:=\frac{\nu_0}{4} \log l>4(M_1+M_2)(1+\psi_1(0))+64,
\end{equation}
where $M_{1}>0$ and $M_{2}>0$ are those in \eqref{defm1} and \eqref{defm2}, respectively.
Then, the functions $\zeta_1$ and $\zeta_2$ are defined as follows, and we have the following
\begin{lem}
\label{zetalem}
We define 
\begin{equation*}
\zeta_1 =
\begin{cases}
\mu_1 b l t^{-\frac{1}{\mu_1 b}}(1-
\frac{(\mu_1 b+1)b}{(\mu_1 b)^3}\frac{\log t+1}{t}) & \text{if $\mu_1>0$,}\\
(2b)^{1/2}e^{-1}l e^{-(2t/b)^{1/2}}t^{1/2} & \text{if $\mu_1=0$}
\end{cases}
\end{equation*}
and
\begin{equation*}
\zeta_2 =
\begin{cases}
 e^{\frac{1}{b\mu_1}-5}t^{-\frac{1}{\mu_1 b}}e^{-\frac{(1-5b\mu_1)\log t}{b^2 \mu_{1}^3 t}}
 & \text{if $\mu_1>0$,}\\
e^{10} t^{-5} e^{-(2t/b)^{1/2}} & \text{if $\mu_1=0$.}
\end{cases}
\end{equation*}
Then, we have
\begin{equation*}
-\frac{\zeta_1}{\zeta_1'}=b\mu_1 t- b\log (-\zeta_{1}')+b\log l+O(t^{-1/2}), \quad \text{as $t\to\infty$}
\end{equation*}
and
\begin{equation*}
-\frac{\zeta_2}{\zeta_2'}= b\mu_1 t- b\log \zeta_2-5b\log t
+O(t^{-1/2}), \quad \text{as $t\to\infty$}.
\end{equation*}
Moreover, if $\mu_1>0$, we have
\begin{equation*}
-\zeta_i'\simeq t^{-\frac{1}{\mu_1 b}-1}, \quad \zeta_i'' \simeq t^{-\frac{1}{\mu_1 b}-2}, \quad  -\zeta_i''' \simeq t^{-\frac{1}{\mu_1 b}-3}
\end{equation*}
for any $i=1,2$ and $t\ge 2$. On the other hand, if $\mu_1=0$, we have
\begin{equation*}
-\zeta_1'\simeq e^{-(2t/b)^{1/2}}, \quad \zeta_1'' \simeq e^{-(2t/b)^{1/2}} t^{-1/2}, \quad  -\zeta_1''' \simeq e^{-(2t/b)^{1/2}} t^{-1}
\end{equation*} 
and
\begin{equation*}
    -\zeta_2'\simeq e^{-(2t/b)^{1/2}}t^{-11/2}, \quad \zeta_2'' \simeq e^{-(2t/b)^{1/2}} t^{-6}, \quad  -\zeta_2''' \simeq e^{-(2t/b)^{1/2}} t^{-13/2}
\end{equation*}
for any $t\ge 2$. Finally, we have 
$$
|(\zeta_2''/\zeta_2')'|=O(t^{-1}|\zeta_2''/\zeta_2'|) \quad \text{as $t\to\infty$.}
$$
\end{lem}
\begin{proof}
We only compute $-\zeta_1/\zeta'_{1}$ for the case $\mu_1>0$. We denote by $k=(\mu_1 b)^{-1}$ and $d=(k+1)bk^2$. Then, we have
$\zeta_1=lk^{-1}t^{-k}(1-d(\log t+1)t^{-1})$. In addition,
by a direct computation, we obtain
\begin{equation*}
-\frac{k}{l}\zeta'_{1}=t^{-k-1}(k(1-d\frac{\log t+1}{t})-d\frac{\log t}{t}), \h \log(- \zeta_{1}')=-(k+1)\log t+\log l+O(\frac{\log t}{t}).
\end{equation*}
Hence, it yields
\begin{align*}
-\frac{\zeta_1}{\zeta'_1}=\frac{t}{k}+\frac{\frac{d}{k}\log t}{k(1-d\frac{\log t+1}{t})-d\frac{\log t}{t}}&=\frac{t}{k}+\frac{d\log t}{k^2}+O(\frac{(\log t)^2}{t})\\
&=\frac{t}{k}-b\log(-\zeta_{1}')+b\log l+O(\frac{(\log t)^2}{t}).
\end{align*}
The other assertions are proved by similar computations.
\end{proof}

Next, we aim to construct a super-solution. We define $\Gamma(r,t)$ and
$\nu_1(t)>0$ as
\begin{equation}
\label{defGam}
\frac{1}{\nu_1 (t)}=\int_{B_{1}^2}(K+\mu_1)\Gamma \psi_1\,dx \quad \text{and}\quad \Gamma=-\log r+ \eta(r/\delta(t))\log(r/\delta(t)),
\end{equation}
where $\eta$ is a cut-off function in \eqref{defeta} and $\delta$ is defined by
\begin{equation}
\label{defdelta}
\delta(t)^4=-e^{-\mu_1 t}\zeta_1 '(t).
\end{equation}
Let $T_0>0$ be a number such that $\delta<1$ for $t>T_0$.
Then, thanks to the fact that $K\in C^{0}(\overline{B_1})$, we can define $\xi_1$ as the unique solution of
\begin{equation*}
\cE \xi_1=-\psi_1+\nu_1(t)(K+\mu_1)\Gamma \quad \text{in $B_{1}^2$,}\quad \xi_1=0 \quad \text{on $\partial B_{1}^2$,} \quad \xi_1\perp \psi_1 \quad \text{in $L^2(B_{1}^2)$}
\end{equation*}
for each $t>T_0$. In addition, we obtain the following
\begin{lem}
\label{nu-1-lem}
It follows that $\nu_1\in C^1(T_{0},\infty)$,
\begin{equation*}
\text{$\nu_1=\nu_0+O(\delta(t)^2)$ \h and \h $\nu_1'(t)=O(\delta(t)|\delta'(t)|)$ \h as $t\to\infty$.}
\end{equation*}
In addition, $\xi_1 \in C^{1}((T_{0},\infty); H^{1}_{0}(B_1))$ and $\xi_1(r,t), \xi_{1,t}(r,t)\in C^1_{0}(\overline{B_1^2})$ for any $t>T_{0}$. Moreover, we have 
\begin{equation*}
    \text{$\lVert \zeta_1\rVert_{C^1(B_{1}^2)}\le C$, \quad $|\xi_1|\le \frac{M_2}{2}\psi_{1}$\quad in $B_1$\quad and \quad$|\xi_{1,t}|<C|\delta'|\delta^{-1/3} \psi_{1}$ 
\quad in $B_1$}
\end{equation*}
for any $t$ sufficiently large with some $C>0$ independent of $t$. 
\end{lem}
\begin{proof}
Thanks to \eqref{pnu0} and \eqref{defGam},  
we obtain the following
\begin{equation*}
\frac{|\nu_1-\nu_0|}{\nu_1\nu_0}\le \int_{B_{2\delta}} (K+\mu_1)|\eta(r/\delta)\log(r/\delta)|\psi_1(r)\le O(\delta^2).
\end{equation*}
In addition, by the definitions of $\nu_1$ and $\Gamma$, we obtain $\nu_1\ge c$ with some $c>0$ independent of $t$ for any sufficiently large $t>0$. Hence, we have $\nu_1=\nu_0+O(\delta^2)$ as $t\to\infty$.

Next, we can deduce from $\Gamma\in C^{1}((T_{0},\infty) ; L^1(B_{1}^2))$ that $\nu_1(t)\in C^1(T_{0},\infty)$. Then, by differentiating the equation with respect to $t$, we have
\begin{equation*}
0=\int_{B_{1}^{2}}(K+\mu_1)\psi_1(\nu_1'(t)\Gamma -\nu_1(t)\delta^{-2}\delta'(r\eta'(r/\delta)\log (r/\delta)+\eta(r/\delta)\delta) \,dr.
\end{equation*}
As a result, we have
\begin{equation*}
|\frac{\nu_1'}{\nu_{1}^2}|\le C\delta|\delta'|.
\end{equation*}
Now, it follows from \eqref{defm1} that
$0\le \Gamma\le -\log r+\log 2$ and $|\Gamma_t|\le C\delta^{-1}|\delta'|\chi_{r<2\delta}$ in $B_{1}$ for any $t>T_0$. From the above estimates, we obtain
\begin{equation*}
\lVert -\psi_{1}+\nu_1(t)(K+\mu_1)\Gamma\rVert_{L^3}\le \frac{M_2}{2\hat{q}}
\end{equation*}
for any $t$ sufficiently large and
\begin{equation*}
\lVert(K+\mu_1)(\nu_1'(t)\Gamma+\nu_1(t)\Gamma_t)\rVert_{L^3}\le O(\delta^{-1/3}|\delta'|)\quad \text{as $t\to\infty$.}
\end{equation*}
In addition, we can deduce that $\xi_{1}\in C^{1}((T_0,\infty); H^{1}_{0}(B_{1}^2))$ and $\xi_{1,t}$ satisfies
\begin{equation*}
\cE \xi_{1.t}= (K+\mu_1)(\nu_1'\Gamma+\nu_1\Gamma_{t}) \quad \xi_{1,t}=0 \quad \text{on $\partial B_{1}^2$} \quad \xi_{1,t}\perp \psi_1 \quad \text{in $L^2(B_{1}^2)$}.
\end{equation*}
Hence, it follows 
from Lemma \ref{well-def-lem} and the definition of $\hat{q}$
that
$\xi_1, \xi_{1,t}\in C^{1}_{0}(\overline{B_1^2})$,
$|\xi_1|\le \frac{M_2}{2}\psi_1$ for any $t$ sufficiently large and $|\xi_{1,t}|=O(\delta^{-1/3}|\delta'|\psi_1)$ as $t\to\infty$.
\end{proof}
Thanks to Lemma \ref{nu-1-lem} and \eqref{defm2}, we deduce that
\begin{equation}
\label{px1}
0\le \xi_1+\frac{M_2}{2} \psi_1\le M_2\psi_1 \quad \text{in $B_1$}
\end{equation}
for all $t$ sufficiently large. Then, we construct a super-solution $\Psi_1$ as follows.
\begin{prop}
\label{superprop}
There exist $\hat{T}>T_0$ such that
\begin{equation*}
\Psi_{1}=\zeta_1 (t)\psi_1+\zeta_{1}'(t)(\nu_1(t)\Gamma+\xi_1+\frac{M_2}{2}\psi_1)\in C^{1}((T_0,\infty); H^{1}_{0}(B_{1}^2))
\end{equation*}
satisfies $\Psi_1=0$ on $\partial B_1\times (\hat{T},\infty)$ and
\begin{equation*}
(\partial_{t}+\cE)\Psi_{1}\ge -\min\{\lambda^{*}e^{u^{*}}r^{4}e^{\mu_1 t}\cF(r^{-4}e^{-\mu_1 t}\Psi_{1}),
-L\zeta_{1}'(t)\delta^{-2}(t)\}
\end{equation*}
in $B_{1}^{2}\times (\hat{T},\infty)$. Moreover, we obtain
\begin{equation*}
\Psi_{1}(r)\ge \frac{-L\zeta'_{1}(t)}{2\psi_1(0)}\psi_1(r) \quad \text{in $B_{1}^{2}$} \quad\text{and}\quad \Psi_{1}(\delta)\le -2L\zeta_{1}'(t)
\end{equation*}
for any $\hat{T}<t$.
\end{prop}

\begin{proof}
By the Hopf lemma, Lemma \ref{zetalem}, \eqref{defdelta} and \eqref{defGam}, we have
\begin{equation}
\label{pGam}
c\psi_1\le \Gamma\le -C\log \delta \psi_1 \quad \text{in $B_1 \times (T_0,\infty)$, }\hspace{2mm} \frac{\zeta_1'(t) \delta'\delta^{-1/3}}{\zeta_1''(t)}\to 0 \quad \text{as $t\to\infty$}
\end{equation}
with some $0<c<C$ independent of $t$.
By Lemma \ref{nu-1-lem} and \eqref{px1}, we obtain
\begin{align*}
(\partial_t+\cE)\Psi_1&=\zeta_{1}''(t)(\nu_1(t)\Gamma+\xi_{1}+\frac{M_2}{2}\psi_1)-\zeta_{1}'(t)\nu_1(t)\Delta \Gamma +\zeta_{1}'(t)\partial_t(\nu_1(t)\Gamma+\xi_1)\\
&= \zeta_{1}''(t)(\nu_1(t)\Gamma+\xi_{1}+\frac{M_2}{2}\psi_1)-\zeta_{1}'(t)(\nu_{0}+O(\delta(t)^2))(\Delta \Gamma -\Gamma_t) \\
&\hspace{60mm}+O(-\zeta'_1(t)|\delta'|\delta^{-1/3})\psi_1)\\
&\ge -\zeta_{1}'(t)(\nu_{0}+O(\delta(t)^2))(\Delta \Gamma -\Gamma_t)
\end{align*}
in $B_{1}^{2}$ for any $t$ sufficiently large. Moreover, since $\delta'(t)\to 0$ as $t\to\infty$, we have
\begin{equation*}
|\Delta \Gamma-\Gamma_t|\le \delta'(t)|J_s|+|\Delta J|\le M_1\nu_{0}^{-1} \delta(t)^{-2}\chi_{r<2\delta(t)}\quad \text{in $B_{1}^2$}
\end{equation*}
for all $t$ sufficiently large.
Hence, we obtain
\begin{equation}
\label{supersiki}
(\partial_t+\cE)\Psi_1\ge 2M_{1}\zeta_{1}'(t)\delta^{-2}(t)\chi_{r<2\delta}\ge L\zeta'_{1}(t)\delta^{-2}(t)\chi_{r<2\delta}\quad \text{in $B_{1}^2$}
\end{equation}
for any $t>T_1$ with some $T_1>0$ because $2M_{1}<L$.

Next, we focus on the behavior of $\Psi_1$. 
By \eqref{defb}, \eqref{defdelta} and Lemma \ref{zetalem}, we have
\begin{equation*}
-\frac{\zeta_1}{\zeta_1'}(\psi_1(0)+O(\delta))+\nu_0\log \delta=L+O(\delta\log\delta)+O(t^{-1/2})=L+O(t^{-1/2})
\end{equation*}
as $t\to\infty$. 
Hence, it follows from Lemma \ref{nu-1-lem} and \eqref{asL}, \eqref{px1}, \eqref{pGam} that 
\begin{align}
\Psi_1&=-\zeta_{1}'(t)(-\frac{\zeta_1}{\zeta'_1}(\psi_1(0)+O(\delta))-(\nu_0+O(\delta^{2}))I(r,t)-(\xi_1+\frac{M_2}{2} \psi_1))\notag\\
&=-\zeta_{1}'(t)(-\frac{\zeta_1}{\zeta'_1}(\psi_1(0)+O(\delta))+\nu_0\log \delta +\nu_0(1-\eta(r/\delta))\log(r/\delta)
-(\xi_1+\frac{M_2}{2} \psi_1))\notag\\
&\quad \quad+O(-\zeta_1'(t) t^{-1/2})\notag\\
&\ge -\frac{L}{2}\zeta_{1}'(t) \label{nagai-1}
\end{align}
in $B_{2\delta}^2$ for any $t$ sufficiently large. In particular, we have
\begin{equation}
\label{nagai-2}
r^{-4} e^{-\mu_1 t}\Psi_1(r)\ge -\frac{L}{2}(2\delta)^{-4} e^{-\mu_1 t}\zeta_{1}'(t)\ge \frac{L}{32}\ge 2 \quad \text{in $B_{2\delta}^2$.}
\end{equation}
Moreover, by a similar computation, we get
\begin{equation*}
\Psi_{1}(\delta(t))\le -2L\zeta'_{1}(t)
\end{equation*}
for any $t$ sufficiently large.
In addition, it follows from Lemmata \ref{zetalem} and \ref{nu-1-lem} that
\begin{equation*}
\Psi_{1}'=\zeta_1(t)\psi_{1}'(r)+\zeta_1'(t)(-\nu_1(t)r^{-1}+\xi_{1}'+\frac{M_2}{2}\psi_{1}') \ge 0 \quad \text{in $B_{t^{-4/3}}^2\setminus B_{2\delta}^2$}
\end{equation*}
and
\begin{equation*}
\Psi_{1}(t)\ge (\zeta_1-O(\zeta_1'\log t))\psi_1\ge \frac{\zeta_1}{2}\psi_1 \quad \text{in $B_{1}^2\setminus B_{t^{-4/3}}^2$}
\end{equation*}
for any $t$ sufficiently large. Therefore, it follows from Proposition \ref{evprop}, 
\eqref{pF}, \eqref{nagai-1} and \eqref{nagai-2} that
\begin{equation*}
\Psi_{1}\ge -\frac{L}{2\psi_1(0)}\zeta_{1}'\psi_1(r) \quad \text{in $B_{1}^2$}
\end{equation*}
and
\begin{align*}
-\lambda^{*}e^{u^{*}}r^{4}e^{\mu_1 t}\cF(e^{-\mu_1 t}r^{-4}\Psi_1) &=-(\frac{2N-4}{r^2}+K)r^4 e^{\mu_1 t}\cF(e^{-\mu_1 t}r^{-4}\Psi_1)\\
&\le
4L r^{-2}\zeta'_1(t) \quad \text{in $B_{2\delta}^2$}
\end{align*}
for any $t>T_{2}$ with some $T_{2}>0$. Therefore, the conclusion follows from \eqref{supersiki}.
\end{proof}
As a result of Proposition \ref{superprop}, we get the following upper estimate of $\Phi$.
\begin{thm}
\label{supercor}
There exists $\tau>0$ such that
\begin{equation}
\label{eq-cor-u-1}
\Psi(r,t,\tau)\le \Psi_{1}(r) \quad \text{in $B_{1}^2$}
\end{equation}
for any $\hat{T}<t$, where $\hat{T}$ is that in Proposition \ref{superprop}. As a result, we have 
\begin{equation}
\label{eq-cor-u-2}
\Phi(\delta(t-\tau),t)\le C
\end{equation}
and
\begin{equation*}
    \Phi\le C e^{-\mu_1 t} (\zeta_1\phi_1+\zeta_1'(t)\nu r^{-4}\Gamma)
\end{equation*}
for any $t>\overline{T}$, where $C>0$ and $\overline{T}>0$ are independent of $t$.
\end{thm}
\begin{proof}
Thanks to \eqref{ue}, one has
\begin{equation*}
  \Psi(r,\hat{T},\tau)=e^{\mu_1 \hat{T}}r^4\Phi(r,\hat{T}+\tau)\to 0\quad \text{as $\tau\to\infty$}
\end{equation*}
in $L^{\infty}(B_{1})\cap C^{1}_{\mathrm{loc}}(B_{1}\setminus\{0\})$.
Then, it follows from Proposition \ref{superprop} that $\Psi(r,\hat{T},\tau)\le \Psi_{1}(r,\hat{T})$ in $B_{1}^2$
for a sufficiently large fixed parameter $\tau$. From the fact and Proposition \ref{superprop}, we obtain \eqref{eq-cor-u-1} by using a comparison principle\footnote{Note that the comparison principle holds since $\cF(\Psi)$ is non-decreasing in $\Psi$.}.

In addition, for any $t>\tau + \hat{T}$, we get
\begin{equation*}
\Phi(r,t)\le e^{-\mu_1 (t-\tau)}r^{-4}\Psi(r,t-\tau,\tau)\le  e^{-\mu_1 (t-\tau)}r^{-4} \Psi_{1}(r,t-\tau)
\end{equation*}
in $B_{1}^2$. In particular, \eqref{eq-cor-u-2} follows from Proposition \ref{superprop} and \eqref{defdelta}.

Now, thanks to Proposition \ref{superprop}, Lemma \ref{zetalem} and \eqref{px1}, we obtain
\begin{align*}
\Psi_1= \zeta_1(t)\psi_1+\zeta_1'(t)\nu_0\Gamma&-\zeta_1'(t)(\nu_0-\nu_1(t))\Gamma\\
&=\zeta_1(t)\psi_1+\zeta_1'(t)\nu_0\Gamma+O(|\zeta_1'| \delta^2\log \delta\psi_1)\\
&\le 2(\zeta_1 (t)\psi_1+\zeta_1'(t)\nu_0\Gamma).
\end{align*}
Thus, it follows from Proposition \ref{evprop} and \eqref{defb} that
\begin{equation*}
e^{-\mu_1 t} r^{-4}\Psi_{1}(r,t)\le C e^{-\mu_1 t} (\zeta_1\phi_1+\zeta_1'(t)\nu r^{-4}\Gamma) \quad \text{in $B_1^2$}
\end{equation*}
for any $t$ sufficiently large with some $C>0$ independent of $t$.
Therefore, it suffices to prove 
\begin{equation*}
\Psi_{1}(r,t-\tau)\le C \Psi_1(r, t) \quad \text{in $B_{1}^2$}
\end{equation*}
for any $t$ sufficiently large, in order to complete the proof. By the mean-value theorem, we have
\begin{equation*}
\Psi_1(r,t-\tau)\le \tau |\Psi_1'(r,t_0)|+\Psi_1(r,t) \quad \text{in $B_{1}^2$}
\end{equation*}
with some $t-\tau\le t_0\le t$. 
By using Lemmata \ref{zetalem}, \ref{nu-1-lem}, Proposition \ref{superprop}, and  \eqref{pGam}, \eqref{defdelta}, \eqref{defm1}, we obtain
\begin{align*}
|\Psi_1'(t_0)|&\le C |\zeta_1'(t_0)|\psi_1+|\zeta_1''(t_0)\nu_1(t_0)+\zeta_1'(t_0)\nu_1'(t_0)||\Gamma(t_0)|\\
&\qquad\qquad+ \nu_1(t_0)|\zeta_1'(t_0)||\Gamma_t|+ |\zeta''_{1}(t_0)|\psi_1+ |\zeta_1'(t)\partial_t \xi_{1}(t_{0})|\\
&\le C |\zeta_1'(t_0)|\psi_1(1+|\frac{\zeta_1''(t_0)}{\zeta_1'(t_0)}\log \delta|+
\delta^{-1}|\delta'|)\le C \Psi_1(r,t) \quad \text{in $B_{1}^2$.}
\end{align*}
Hence, we obtain the assertion.
Therefore, the proof is completed.
\end{proof}
Next, we construct a sub-solution. We define
\begin{equation}
\label{defnu2}
\frac{1}{\nu_2(t)}:=\int_{B_{1}^{2}}(\frac{2\zeta_2''}{\zeta_{2}'}-(K+\mu_1))\log r \psi_1\,dx=\frac{1}{\nu_0}+\int_{B_{1}^{2}}\frac{2\zeta_2''(t)}{\zeta_{2}'(t)}\log r \psi_1\,dx.
\end{equation}
Then, for any fixed $t>2$, we can define $\xi_2\in H^{1}_{0}(B_{1}^2)$ as the unique solution to
\begin{equation*}
\cE \xi_2=-\psi_1-\nu_2(t)(K+\mu_1)\log r+2\nu_2(t)\frac{\zeta_2''(t)}{\zeta_2'(t)}\log r\hspace{2mm} \text{in $B_{1}^2$,} \quad \xi_2=0 \hspace{2mm} \text{on $\partial B_{1}^2$}
\end{equation*}
with $\xi_2\perp \psi_1$ in $L^2(B_{1}^2)$. We first introduce the following.
\begin{lem}
\label{nu-2-lem}
It follows that $\nu_2\in C^1(1,\infty)$ and
\begin{equation*}
    \text{$\nu_2=\nu_0 +O(|\zeta_2''/\zeta_2'|)$ \h and \h $|\nu_2'|=O(t^{-1}|\zeta_2''/\zeta_2'|)$\h as $t\to\infty$. }
\end{equation*}
In addition, $\xi_2\in C^{1}((1,\infty); H_{0}^1 (B_1^2))$ and $\xi_2 (r,t), \xi_{2,t}(r,t)\in C^{1}_{0}(\overline{B_{1}^2})$ for any $t>1$. Moreover, we have $\lVert \xi_2\rVert_{C^1(B_{1}^2)}<C$ with some $C>0$ independent of $t$. Finally, $|\xi_2|\le C\psi_1$ in $B_{1}^2$ for any $t$ sufficiently large and 
\begin{equation*}
\lVert \xi_2\rVert_{C^1(B_{1}^2)}<C, \quad |\xi_2|\le C\psi_1 \quad \text{in $B_{1}^2$}\quad  \text{and}\quad |\xi_{2,t}|\le -Ct^{-1}\frac{\zeta_2''}{ \zeta_2'}\psi_1 \quad \text{in $B_{1}^2$}
\end{equation*}
for any sufficiently large $t$,
where $C>0$ is independent of $t$.
\end{lem}
\begin{proof}
It easily follows from \eqref{defnu2} that $\nu_2=\nu_0+O(|\zeta_2''/\zeta_2'|)$ and $|\nu_2'|=O(|(\zeta_2''/\zeta_2')'|)$ as $t\to\infty$. By Lemma \ref{zetalem}, we obtain $|\nu_2'|=O(t^{-1}|\zeta_{2}''/\zeta_2'|)$ as $t\to\infty$. Hence, it follows that
\begin{equation*}
    \lVert -\psi_1-\nu_2(t)(K+\mu_1)\log r+2\nu_2(t)\frac{\zeta_2''(t)}{\zeta_2'(t)}\log r\rVert_{L^3(B_{1}^{2})}<C
\end{equation*}
and
\begin{equation*}
\lVert -\nu_2'(t)(K+\mu_1)\log r+2(\nu_2'(t)\frac{\zeta_2''(t)}{\zeta_2'(t)}+\nu_2(\zeta_2''/\zeta_2')')\log r\rVert_{L^3(B_{1}^2)}<C t^{-1}|\zeta_2''/\zeta_2'|
\end{equation*}
with some $C>0$ independent of $t$.
In addition, we verify that $\xi_2\in C^{1}((1,\infty); H_{0}^1 (B_{1}^2))$ and $\xi_{2,t}$ satisfies
\begin{equation*}
\cE \xi_{2,t}=-\nu_2'(t)(K+\mu_1)\log r+2(\nu_2'(t)\frac{\zeta_2''(t)}{\zeta_2'(t)}+\nu_2(\zeta_2''/\zeta_2')')\log r,
\end{equation*}
$\xi_{2,t}=0$ on $\partial B_{1}$ and $\xi_{2,t}\perp \psi_1$ in $L^2(B_{1}^2)$. By Lemma \ref{well-def-lem}, we obtain $\lVert \xi_2\rVert_{C^1(B_{1}^2)}=O(1)$,
$\xi_2, \xi_{2,t}\in C^1_{0}(\overline{B_1^2})$, $|\xi_2|=O(\psi_1)$ and $|\xi_{2,t}|=O(-t^{-1}(\zeta_2''/\zeta_2')\psi_1)$ as $t\to\infty$.
\end{proof}
Thanks to Lemma \ref{nu-2-lem}, we can choose $l_2>0$ so that 
\begin{equation}
\label{px2}
  -\frac{3 l_2}{2}\psi_1 \le \xi_2-l_2\psi_1\le  -\frac{l_2}{2} \psi_1 \quad \text{in $B_{1}^2$}
\end{equation}
for any $t$ sufficiently large. We define
\begin{equation}
\label{equPsi}
\uPsi=\zeta_2\psi_1+\zeta'_2 (t)(-\nu_2(t)\log r+\xi_2-l_2\psi_1).
\end{equation}
Then, we introduce the following properties of $\uPsi$.
\begin{lem}
\label{rholem}
There exist $T_{*}>0$ and $C>0$ independent of $t$ such that $\uPsi$ has the unique zero $\rho=\rho(t)$ for any $t>T_{*}$. More precisely, $\uPsi<0$ in $(0,\rho)$ and $\uPsi>0$ in $(\rho,1)$.
In addition, there exists $C>0$ independent of $t$ such that
\begin{equation}
\label{prho}
-\log \rho \le \frac{\mu_1 t}{4}-\frac{1}{4}\log \zeta_2-\frac{5}{4}\log t+C
\end{equation}
and
\begin{equation*}
\uPsi(r)<\frac{\nu_0}{2}\zeta_2'<0\quad\text{in $B_{\rho/2}^2$,} \quad e^{-\mu_1 t}r^{-4}(\uPsi(r)- \frac{\nu_0}{4\psi_1(0)}\zeta_2'\psi_1)\le Ct^{-5}\psi_1(r)\quad\text{in $B_{1}^2$}
\end{equation*}
for any $t>T_{*}$. 
\end{lem}
\begin{proof}
By differentiating \eqref{equPsi} with respect to $r$ and using Lemmata \ref{zetalem} and \ref{nu-2-lem}, we obtain
\begin{equation}
\label{hojo-1}
\uPsi'=\zeta_{2}\psi_{1}'(r)+\zeta_{2}'(t)(-\nu_2(t)r^{-1}+\xi_{2}'-l_2\psi_{1}')>0 \quad \text{in $B_{(\zeta_2'/\zeta_2)^2}^2$}
\end{equation}
for any $t$ sufficiently large. In addition, by Lemma \ref{zetalem}, we have
\begin{equation*}
\uPsi\ge \bigl(\zeta_2(t)-O(|\zeta_{2}'\log( -\zeta'_{2}/\zeta_2)|)\bigr) \psi_{1}\ge \bigl(\zeta_2-O(|\zeta_{2}'\log t|)\bigr)\psi_1\ge \frac{\zeta_{2}}{2}\psi_1 \quad \text{in $B_{1}^2 \setminus B_{(\zeta_2'/\zeta_2)^2}^2$}
\end{equation*}
for all $t$ sufficiently large.
Hence, there exists the unique zero $\rho$ of 
$\uPsi$ satisfying $0<\rho<(\zeta_{2}'/\zeta_2)^2$. 
Then, it follows from Lemmata \ref{zetalem}, \ref{nu-2-lem} and \eqref{defb} that
\begin{equation*}
-\log \rho =\frac{1}{(\nu_0+O(|\zeta_{2}''/\zeta_{2}'|))}(-\frac{\zeta_{2}}{\zeta'_{2}}(\psi_{1}(0)+O(\rho))+O(1))\le \frac{\mu_1 t}{4}-\frac{1}{4}\log \zeta_2 -\frac{5}{4}\log t+C
\end{equation*}
for any $t$ sufficiently large. Hence, we obtain \eqref{prho}. 
It follows from Lemma \ref{zetalem} and \ref{nu-2-lem} that
\begin{equation}
\label{atarimae}
    \uPsi\le \zeta_2 \psi_1+\zeta_2'(t)(\xi_2-l_2\psi_1)\le 2\zeta_2 \psi_1 \quad \text{in $B_{1}^{2}$}
\end{equation}
for any $t$ sufficiently large. Moreover,
thanks to \eqref{hojo-1} and Lemmata \ref{nu-2-lem} and \ref{zetalem}, we get
\begin{align*}
\uPsi(r)\le \uPsi(\frac{\rho}{2})-\uPsi(\rho)&=\zeta_2(\psi_{1}(\rho/2)-\psi_{1}(\rho))+\nu_2(t)\zeta_2'(t)\log 2\\
&\quad+\zeta'_2(t)(\xi_2(\rho/2, t)-\xi_2(\rho,t))-
l_2\zeta_2'(t)(\psi_1(\rho/2)-\psi_1(\rho))\\
&=\zeta'_{2}(t) (\nu_0+O(|\zeta_2''/\zeta_2'|))\log 2 +O(\zeta_2\rho)\le \frac{\nu_0}{2}\zeta'_{2}
\end{align*}
in $B_{\rho/2}^2$ for any $t$ sufficiently large. By combining the above estimate
and \eqref{prho}, \eqref{atarimae}, we have
\begin{equation*}
e^{-\mu_1 t}r^{-4}(\uPsi(r)- \frac{\nu_0}{4\psi_1(0)}\zeta_2'\psi_1)\le 16
e^{-\mu_1 t}\rho^{-4}(\uPsi(r)- \frac{\nu_0}{4\psi_1(0)}\zeta_2'\psi_1)\le -Ct^{-5}\psi_{1}(r) \h \text{in $B_{1}^2$}
\end{equation*}
for any $t$ sufficiently large with some $C>0$ independent of $t$.
\end{proof}
We define
\begin{equation*}
R:=\lambda^{*}e^{u^{*}}r^{4}e^{\mu_1 t}\cF(r^{-4}e^{-\mu_1 t}(\uPsi-\frac{\nu_0}{4\psi_1(0)}\zeta'_2(t)\psi_1))  
\end{equation*}
and
\begin{align*}
&G(r,t):=\partial_t R=
\mu_1\lambda^{*}e^{u^{*}+\mu_1 t}r^4\cF(r^{-4}e^{-\mu_1 t}(\uPsi-\frac{\nu_0}{4\psi_1(0)}\zeta'_{2}\psi_1))\\
+&\lambda^{*}e^{u^{*}}(-\mu_1 (\uPsi-\frac{\nu_0}{4\psi_1(0)}\zeta'_{2}\psi_1)+\uPsi_t-\frac{\nu_0}{4\psi_1(0)}\zeta''_{2}\psi_1)\cF'(r^{-4}e^{-\mu_1 t}(\uPsi-\frac{\nu_0}{4\psi_1(0)}\zeta_2'\psi_1)).
\end{align*}
Then, thanks to Lemma \ref{rholem} and the definition of $\cF$, we have $\mathcal{F}=\mathcal{F'}=0$ on $B_{\rho(t)/2}$. Hence, by Lemma \ref{nu-2-lem}, we obtain $R\in W^{1,\infty}_{\mathrm{loc}}((1,\infty); L^{2}(B_{1}))$ and $R_t=G$ in $B_{1}^{2}$ for a.e. $t>1$. Hence, there exists a unique solution $\upsi\in W^{1,\infty}_{\mathrm{loc}}((1,\infty); H^{1}_{0}(B_{1}^2))$ of
\begin{equation*}
\cE \upsi= (R, \psi_1)_{L^2(B_{1}^2)}\psi_{1}-R \quad \text{in $B_{1}^2$,} \quad \text{$\upsi=0$ on $\partial B_{1}^{2}$} \quad \text{$\upsi\perp \psi_{1}$ in $L^{2}(B_{1}^2)$}.
\end{equation*}
Moreover, $\upsi_t$ satisfies
\begin{equation*}
\cE \upsi_t = (G, \psi_1)_{L^2(B_{1}^2)}\psi_{1}-G \quad \text{in $B_{1}^2$,} \quad \text{$\upsi_t =0$ on $\partial B_{1}^{2}$} \quad \text{$\upsi_t \perp \psi_{1}$ in $L^{2}(B_{1}^2)$}.
\end{equation*}
for a.e. $t>1$.
Then, we obtain the following
\begin{lem}
\label{upsilem}
It follows that
$\upsi\in C^{1}_{0}(\overline{B_{1}^2})$ for any sufficiently large $t$ and $\upsi_t \in C^{1}_{0}(\overline{B_{1}^2})$ for a.e. sufficiently large $t$. In addition, we have
\begin{equation*}
\lVert \upsi\rVert_{C^{1}(B_{1}\setminus B_{1/2})}\le C \zeta_2 t^{-5/2}\quad\text{and} \quad |\upsi|\le C \zeta_2 t^{-5/2}\psi_1 \quad \text{in $B_{1}^2$}
\end{equation*}
for any $t$ sufficiently large and
\begin{equation*}
    \lVert \upsi_{t}\rVert_{C^{1}(B_{1}\setminus B_{1/2})}\le C \zeta_2 t^{-5/2}\quad\text{and} \quad |\upsi_{t}|\le C\zeta_2 t^{-5/2}\psi_1 \quad \text{in $B_{1}^2$}
\end{equation*}
for a.e. $t$ sufficiently large, where $C>0$ is independent of $t$. 
\end{lem}
\begin{proof}
By Lemmata \ref{rholem} and \ref{zetalem}, we obtain $R=G=0$ in $B_{\rho/2}^2$ and
\begin{equation*}
     e^{-\mu_1 t}r^{-4}(\uPsi(r)- \frac{\nu_0}{4}\zeta_2'\psi_1)\le -Ct^{-5}\psi_1(r)\le 2 \quad \text{in $B_{1}^2$}
\end{equation*}
for any $t$ sufficiently large. As a result, by the elliptic regularity,
we have $\psi\in C^{1}_{0}(B_{1}^{2})$ for a.e. $t$ sufficiently large and $\psi\in C^{1}_{0}(B_{1}^2)$. Moreover, by using the above estimate, Proposition \ref{prop-2-1}, \eqref{pF}, \eqref{hojo-1}, \eqref{atarimae} and
Lemmata \ref{zetalem}, \ref{rholem}, we have
\begin{equation}
\label{R-2}
   \lVert r^2(-\log r)^{5/2} R\rVert_{L^{\infty}(B_{1}^2)}\le C (\log (\rho/2))^{5/2} r^{-4}e^{-\mu_1 t}(\uPsi-\frac{\nu_0}{4}\zeta_2'\psi_1)^2
\le C t^{-5/2}\zeta_2(t)
\end{equation}
for any large $t$ with some $C>0$ independent of $t$.
As a result, it yields
\begin{equation}
\label{R-3}
\lVert R\rVert_{L^1}\le C\lVert r^2(-\log r)^{5/2} R\rVert_{L^{\infty}(B_{1}^2)} \int_{B_{1}^2}r^{-2}(-\log r)^{-5/2}\,dx \le C t^{-5/2}\zeta_2(t)
\end{equation}
and
\begin{equation}
\label{R-1}
0\le (R,\psi_{1})_{L^2(B_{1}^2)}\le C\lVert R\rVert_{L^{1}(B_{1}^2)}\le 
C t^{-5/2} \zeta_2(t) 
\end{equation}
for any $t$ sufficiently large with some $C>0$ independent of $t$.
Moreover, it follows from \eqref{atarimae}, 
Lemmata \ref{zetalem},
\ref{nu-2-lem} and \ref{rholem} that
\begin{equation*}
\lVert \uPsi_t \chi_{r>\frac{\rho}{2}}\rVert_{L^{\infty}}\le O(|\zeta_{2}'|)+O(-\zeta_{2}''\log \rho)\le C |\zeta_{2}'|
\end{equation*}
for a.e. $t$ sufficiently large and
\begin{equation*}
    \lVert \uPsi(r)\chi_{r>\frac{\rho}{2}}\rVert_{L^{\infty}}\le O(\zeta_2(t))+O(\zeta_2'(t)\log\rho)\le C \zeta_2 
\end{equation*}
for any $t$ sufficiently large, where $C>0$ is a constant independent of $t$.
Thus, by a similar computation, we obtain
\begin{equation}
\label{G-1}
\lVert r^2(-\log r)^{5/2} G \rVert_{L^{\infty}(B_{1}^2)},\h \lVert G\rVert_{L^1(B_{1}^2)}, \h |(G,\psi_{1})_{L^2(B_{1}^2)}|\le Ct^{-5/2}\zeta_2(t) 
\end{equation}
for a.e. $t$ sufficiently large with some $C>0$ independent of $t$.
Now, by the energy estimate, we get
\begin{align*}
(\mu_2-\mu_1)\lVert \upsi \rVert_{L^2}^2 &\le (R,\psi_1)_{L^2}\lVert \psi_1 \rVert_{L^2}\lVert \upsi\rVert_{L^2}+ \lVert R\rVert_{L^1}\lVert \upsi\rVert_{L^{\infty}}\\
&\le \frac{\mu_2-\mu_1}{4}\lVert \upsi\rVert_{L^2}^2+\frac{(R,\psi_1)_{L^2}^2}{(\mu_2-\mu_1)}+\lVert R\rVert_{L^1}\lVert \upsi\rVert_{L^{\infty}}.
\end{align*}
Moreover, by using Lemma \ref{apenlem}, we have 
\begin{equation*}
\lVert \upsi\rVert_{L^{\infty}}\le C(\lVert \upsi\rVert_{L^2}+ \lVert (R,\psi_1)_{L^2}\psi_1\rVert_{L^{\infty}}+\lVert r^2(-\log r)^{5/2} R\Vert_{L^{\infty}}).
\end{equation*}
Therefore, by using \eqref{R-2}, \eqref{R-3}, \eqref{R-1},
it follows from the Cauchy–Schwarz inequality and an absorbing argument that
\begin{equation}
\label{R-4}
\lVert \upsi\rVert_{L^{\infty}}\le C \zeta_2 t^{-5/2}
\end{equation}
for any $t$ sufficiently large with some $C>0$ independent of $t$.
Moreover, since 
\begin{equation*}
\lVert R\rVert_{L^{\infty}(B_{1}^2\setminus B_{1/2}^2)}=O(e^{-\mu_1 t}\zeta_2^2)\quad \text{as $t\to\infty$}, 
\end{equation*}
it follows from the standard elliptic regularity theory, \eqref{R-4} and the Hopf lemma that 
\begin{equation*}
\lVert \upsi\rVert_{C^{1}(B_{1}\setminus B_{1/2})}\le C t^{-5/2} \zeta_2 \quad\text{and} \quad |\upsi|\le C t^{-5/2}\zeta_2 \psi_1 \quad \text{in $B_{1}^2$}
\end{equation*}
for any $t$ sufficiently large with some $C>0$ independent of $t$.

Thanks to \eqref{G-1} and the fact that
\begin{align*}
\cE \upsi_{t}= (G,\psi_1)_{L^2(B_1^2)}\psi_1-G \quad \text{in $B_{1}^2$,} \quad \upsi_{t}=0 \quad \text{on $\partial B_{1}^{2}$},
\end{align*}
by using a similar argument, we obtain 
\begin{equation*}
\lVert \upsi_{t}\rVert_{C^{1}(B_{1}\setminus B_{1/2})}\le C t^{-5/2}\zeta_2 \quad\text{and} \quad |\upsi_{t}|\le C t^{-5/2}\zeta_2 \psi_1 \quad \text{in $B_{1}^2$}
\end{equation*}
for a.e. $t$ sufficiently large with some $C>0$ independent of $t$. Hence, the conclusion follows.
\end{proof}
We now construct a sub-solution $\Psi_2$ as follows.
\begin{prop}
\label{subprop}
There exists $\Tilde{T}>0$ independent of $\ep$
such that for any $\ep>0$, we have
\begin{equation*}
\text{$\Psi_2:=\uPsi+\upsi\in W^{1,\infty}_{\mathrm{loc}}((0,\infty);H^{1}(B_{1}^2\setminus B_{\ep}^2 ))$}
\end{equation*}
and 
\begin{equation*}
(\partial_t+\cE)\Psi_2\le -\lambda^{*}e^{u^{*}}r^4 e^{\mu_1 t}\cF(r^{-4}e^{-\mu_1 t}\Psi_2) \quad \text{in $B_{1}^2\setminus B_{\ep}^2 \times (\Tilde{T},\infty)$.}
\end{equation*}
In addition, we have
$\Psi_{2}=0$ on $\partial B_{1}^2 \times (\Tilde{T},\infty)$, and $\Psi_2<0$ on $B_{\rho(t)/2}^2 \times [\Tilde{T},\infty)$.
\end{prop}
\begin{proof}
Let $0<\ep<1$. Then, by directly computing and
using Lemmata \ref{nu-2-lem}, \ref{zetalem} and \eqref{px2}, we have
\begin{equation*}
(\partial_t+\cE)\uPsi=\zeta_2''(t)(\xi_2-l_2\psi_1+\nu_2 \log r)-\zeta_{2}'(t)\nu_2'(t)\log r+\zeta'_{2}\xi_{2,t}\le -\frac{l_2}{2} \zeta''_{2}(t)\psi_1
\end{equation*}
in $B_{1}\setminus B_{\ep}$ for any $t>T_1$ with some $T_1>0$ independent of $\ep$. Moreover, it follows from Lemmata \ref{upsilem}, \ref{zetalem} and \eqref{pF} that
\begin{equation*}
    \Psi_2=\uPsi+\upsi=\uPsi+O(\zeta_2 t^{-5/2}\psi_1)\le \uPsi-\frac{\nu_0}{4\psi_1(0)}\zeta_2'\psi_1 \quad \text{in $B_1$}
\end{equation*}
and thus
\begin{equation*}
\lambda^{*}e^{u^{*}+\mu_1 t} \cF(r^{-4}e^{-\mu_1 t}(\uPsi+\upsi))\le R.
\end{equation*}
As a result, it follows from Lemmata \ref{rholem}, \ref{zetalem} and \ref{upsilem} that $\Psi_2(r)<0$ in $B_{\rho(t)/2}$ for ant $t>T_2$ and
\begin{align*}
(\partial_t+\cE)\upsi&=\upsi_t+(R,\psi_{1})_{L^2}\psi_1-R\le O(t^{-5/2}\zeta_2 \psi_1)-R=O(t^{-1/2}\zeta_2''\psi_1)-R\\
&\le \frac{l_2}{4}\zeta_2''(t) \psi_1 -\lambda^{*}e^{u^* + \mu_1 t}r^4  \cF(r^{-4}e^{-\mu_1 t}(\uPsi+\upsi)) \quad \text{in $B_{1}^2$}
\end{align*}
for a.e. $t>T_2$ with some $T_2>0$ independent of $\ep$. Hence, we obtain the assertion by taking $\Tilde{T}\ge \max\{T_1, T_2\}$.
\end{proof}
As a result of Proposition \ref{subprop}, we obtain the following lower estimate of $\Phi$.
\begin{thm}
\label{subcor}
There exists $\underline{T}>0$ such that
\begin{equation}
\label{sub-cor-eq1}
\Psi_2 (r,t)\le \Psi(r, t, 1-\underline{T}) \quad \text{in $B_{1}^2$}
\end{equation}
for any $t>\underline{T}$. As a result, there exist $C>0$, $M>0$ and  $T^{*}>0$ independent of $t$ such that
\begin{equation}
\label{sub-cor-eq2}
\Phi(r,t)\ge \frac{1}{C}e^{-\mu_1 t}(\zeta_2(t)\phi_1-\nu\zeta_2'(t)\log r)+Me^{-\mu_1 t}\zeta_2'(t)\phi_1 \quad \text{in $B_1$}
\end{equation}
for any $t>T^{*}$.
\end{thm}
\begin{proof}
We remark that
\begin{equation*}
\Psi(r, t, 1-t)= r^{4}e^{\mu_1 t}\Phi(r,1) \quad \text{in $[0,1]$ \quad for any $t>0$}.
\end{equation*}
By \eqref{sita}, we have $\Phi(r,1)\ge c\psi_1$ in $[0,1]$
for some $c>0$. In addition, it follows from Lemmata \ref{rholem}, \ref{upsilem} and Proposition \ref{subprop} 
that there exists $\underline{T}>\Tilde{T}$ such that
\begin{equation*}
r^{-4}e^{-\mu_1\underline{T}}\Psi_2(r,\underline{T}) \le C \underline{T}^{-5}\psi_1 +(\frac{\rho(\underline{T})}{2})^{-4}e^{-\mu_1\underline{T}}\upsi\le C\underline{T}^{-5}\psi_1\le c \psi_1 \quad \text{in $B_{1}^2$},
\end{equation*}
where $C>0$ is independent of $\underline{T}$.
Hence, we have 
\begin{equation}
\label{sitakara-hyouka}
\text{$\Psi_2(r,\underline{T})\le \Psi(r, \underline{T}, 1-\underline{T})$ \quad in $B_{1}^2$}.
\end{equation}
We fix $T>\underline{T}$. Then, there exists $\rho_0>0$ such that $\rho_0\le \frac{\rho(t)}{2}$ for $\underline{T}\le t\le T$. By Proposition \ref{subprop}, we have $\Psi_2\le 0$ on $B_{\rho_0}\times [\underline{T},T]$. Thanks to the fact, \eqref{sitakara-hyouka}, Proposition \ref{subprop}, we obtain \eqref{sub-cor-eq1} by using a comparison principle and then letting $T\to\infty$ and $\rho_0\to 0$.

Next, we prove \eqref{sub-cor-eq2}. By Lemma \ref{sub-cor-eq1} and the definition of $\Psi(r,t,\tau)$, we have 
\begin{align}
\label{kumiawase-1}
\Phi(r,t)=e^{-\mu_1(t+\tau)}r^{-4}\Psi(r,t+\tau,-\tau) \ge r^{-4} e^{-\mu_1(t+\tau)} \Psi_2(r,t+\tau)
\end{align}
for any $t\ge 1$, where $\tau=\underline{T}-1$. In addition,
it follows from Lemmata \ref{zetalem}, \ref{nu-2-lem} and \ref{upsilem} that
\begin{align*}
\Psi_2(r,t)&= \zeta_2(t)\psi_1-\nu_2(t) \zeta_2'(t)\log r+
O(-\zeta_2'(t)\psi_1)+
O(\zeta_2(t) t^{-5/2}\psi_1)\\&=\zeta_2(t)\psi_1-\nu_2(t)\zeta_2'(t)\log r+O(-\zeta_2'(t)\psi_1) \quad \text{in $B_{1}^2$ \quad as $t\to\infty$.}
\end{align*}
By the above estimate, Lemmata \ref{nu-2-lem}, \ref{zetalem} and
the mean-value theorem, we have
\begin{align*}
\zeta_2(t)\psi_1-\zeta_2'(t)\nu_2(t)\log r&=\zeta_2(t+\tau)\psi_1-\zeta_2'(t+\tau)\nu_2(t+\tau)\log r\\
&\quad-\tau(\zeta_2'(t_0)\psi_1(r)-(\zeta_2''(t_0)\nu_2(t_0)+\zeta_2'(t_0)\nu_2'(t_0))\log r)\\
&\le \zeta_2(t+\tau)\psi_1-\zeta_2'(t+\tau)\nu_2(t+\tau)\log r\\
&\quad -\tau \zeta_2'(t_0)\psi_1+\tau\zeta''_2(t_0)\nu_0(1-O(t^{-1/2}))\log r\\
&\le \Psi_2(t+\tau) -C(\zeta_2'(t_0)+\zeta_2'(t+\tau))\psi_1 
\end{align*}
in $B_{1}^2$ with some $t<t_0<t+\tau$, where $C>0$ is independent of $t$. 
Therefore, it follows from Lemma \ref{zetalem} 
that
\begin{equation}
\label{kumiawase-2}
\Psi_2(r,t+\tau)\ge \zeta_2(t)\psi_1-\zeta_2'(t)\nu_2(t)\log r+C_1 \zeta_2'(t)\psi_1 \quad \text{in $B_{1}^2$}
\end{equation}
for any $t$ sufficiently large with some $C_1>0$ independent of $t$.

We define $0<r_1(t)<1$ such that 
\begin{equation*}
-\log r_1 =-\frac{2\zeta_2}{\nu_2(t) \zeta_2'}\psi_1(0).
\end{equation*}
Then, thanks to Lemma \ref{nu-2-lem}, it follows that
\begin{equation*}
\zeta_2(t)\psi_1-\zeta_2'(t)\nu_0 \log r+C_1 \zeta_2'(t)\psi_1<0  \quad \text{in $B_{r_1(t)}$}
\end{equation*}
and
\begin{equation*}
 \zeta_2(t)\psi_1-\zeta_2'(t)\nu_2(t)\log r+C_1 \zeta_2'(t)\psi_1<0 \quad \text{in $B_{r_1(t)}$}   
\end{equation*}
for any $t$ sufficiently large.
On the other hand, by using Lemma \ref{nu-2-lem} again, we have
\begin{align*}
    \zeta_2(t)\psi_1-\zeta_2'(t)\nu_0\log r&\le \zeta_2\psi_1-\zeta_2'\nu_2(t)\log r+|(\nu_2(t)-\nu_0)|\zeta_2'\log r\\
    &\le  \zeta_2\psi_1-\zeta_2'\nu_2(t)\log r - \zeta_2''\log r_1 \psi_1\\
 & \le \zeta_2\psi_1-\zeta_2'\nu_2(t)\log r-C_2 \zeta_2'\psi_1
\end{align*}
in $B_{1}^2\setminus B_{r_1}^2$ for any sufficiently large $t$ with some $C_2>0$ independent of $t$. Therefore, 
\begin{equation*}
 \Bigl(\zeta_2(t)\psi_1-\zeta_2'(t)\nu_2(t)\log r+C_1 \zeta_2'(t)\psi_1 \Bigr)_{+}\ge  \Bigl(\zeta_2(t)\psi_1-\zeta_2'(t)\nu_0\log r+(C_1+C_2)\zeta_2'(t)\psi_1\Bigr)_{+}
\end{equation*}
in $B_{1}^2\setminus B_{r_1}^2$ for any sufficiently large $t$.
By combining the above, \eqref{kumiawase-1} and \eqref{kumiawase-2}, and \eqref{sita}, we verify that
\begin{equation*}
\Phi(r,t)\ge e^{-\mu_1(t+\tau)} r^{-4}\Bigl(\zeta_2(t)\psi_1-\zeta_2'(t)\nu_0\log r+C\zeta_2'(t)\psi_1\Bigr)_{+}
\end{equation*}
for any $t$ sufficiently large.
Thus, by using Proposition \ref{evprop} and \eqref{defb}, it yields
\begin{equation*}
\Phi(r,t)\ge \frac{1}{C}e^{-\mu_1t}(\zeta_2(t)\phi_1-\nu\zeta_2'(t)\log r)+Me^{-\mu_1 t}\zeta_2'(t)\phi_1
\end{equation*}
for any $t>T^{*}$ with some $T^{*}>0$, $C>0$ and $M>0$ independent of $t$. Hence, we obtain \eqref{sub-cor-eq2}.
\end{proof}
As a consequence of Theorems \ref{supercor}, \ref{subcor} and Lemma \ref{zetalem}, we establish Theorem \ref{intro-thm-3}.
\subsection{Proof of Theorem \ref{intro-thm-2}}
In this subsection, we prove Theorem \ref{intro-thm-2} by a matching argument. We define $\lVert u\rVert_{L^{\infty}}=\alpha(t)$. We quote the following proposition obtained in \cite[Proposition 4.2]{KO26}.
\begin{prop}
\label{prop-inner}
There exists $\hat{r}>0$ depending only on $f$ such that
For any $\ep\in (0,\hat{r})$, there exist $T>0$ and $C>0$ such that 
\begin{align*}
u^{*}-u\le
Ce^{-2\alpha(t)} r^{-4}\log (re^{\alpha(t)/2})
\quad \text{in $B_{\hat{r}}\setminus B_{\ep}$}
\end{align*}
for any $r\in\left[\ep,\hat{r}\right)$ and
$t>\underline{T}$.
\end{prop}

\begin{proof}[Proof of Theorem \ref{intro-thm-2}.]
Since $\alpha(t):=\lVert u\rVert_{L^{\infty}}$, it follows from  Theorem \ref{supercor} and Proposition \ref{prop-2-1} that
\begin{equation*}
-2\log (\delta(t-\tau))-\alpha(t)-C \le \Phi(\delta(t-\tau),t)\le C \quad\text{for any $t$ sufficiently large}
\end{equation*}
with some $C>0$ independent of $t$.
Moreover, by \eqref{defdelta} and Lemma \ref{zetalem}, we have
\begin{equation*}
-2\log (\delta(t-\tau))=\frac{\mu_1 (t-\tau)}{2}- \frac{1}{2}\log (-\zeta'_{1}(t-\tau))=\frac{\mu_1 t}{2}- \frac{1}{2}\log (-\zeta'_{1}(t))+O(1)
\end{equation*}
as $t\to\infty$.
As a result, it follows that
\begin{equation*}
\frac{\mu_1 t}{2}-\frac{1}{2}\log (-\zeta_{1}'(t))-C\le \alpha(t)
\end{equation*}
for any $t$ sufficiently large with some $C>0$ independent of $t$.
On the other hand, we claim that
\begin{align}
\label{alphaclaim}
\begin{split}
\alpha(t)&\le \frac{\mu_1 t}{2}-\frac{1}{2}\log (-\zeta_{1}'(t))+C \quad\text{if $\mu_1>0$,}\\
\alpha(t)&\le -\frac{1}{2}\log (-\zeta_{1}'(t))+\frac{11}{4}\log t+C \quad \text{if $\mu_1=0$}
\end{split}
\end{align}
for any $t$ sufficiently large with some $C>0$ independent of $t$.
Indeed, we fix $0<\ep<\hat{r}$. Then, 
it follows from Proposition \ref{prop-inner} and Theorem \ref{subcor} that
\begin{equation*}
\frac{1}{C}e^{-\mu_1 t}\ep^{-4}\zeta_{2}(t)\le \Phi(r,t)\le C \alpha e^{-2\alpha}\ep^{-4}
\end{equation*}
for any $t$ sufficiently large.
It implies that
\begin{equation}
\label{main-eq-1}
\mu_1 t-\log \zeta_2(t)\ge 2\alpha-\log \alpha - C
\end{equation}
for any $t$ sufficiently large.
We first consider the case $\mu_1>0$. In this case, it follows from \eqref{main-eq-1} that
\begin{equation*}
\alpha\le \mu_1 t \quad \text{for any $t$ sufficiently large.}
\end{equation*}
Thus, we obtain
\begin{equation*}
\alpha\le \frac{\mu_1 t}{2}+\frac{1}{2}(\log \alpha-\log \zeta_2)+C\le \frac{\mu_1 t}{2}-\frac{1}{2}\log (\zeta_2/t)+C
\end{equation*}
for any sufficiently large $t$.
Therefore, the result follows from the fact that
\begin{equation*}
\log (\zeta_2/t)=\log (-\zeta_1')+O(1) \h \text{as $t\to\infty$}\quad\text{(see Lemma \ref{zetalem})}.
\end{equation*}
We now consider the case $\mu_1=0$. In this case, it follows from \eqref{main-eq-1} and Lemma \ref{zetalem} that
\begin{equation*}
\log \alpha \le \log (-\log \zeta_2)+C\le \frac{1}{2}\log t+C
\end{equation*}
for any sufficiently large $t$.
Hence, it follows from \eqref{main-eq-1} that
\begin{equation*}
\alpha\le \frac{1}{2}(\log \alpha-\log \zeta_2)+C\le -\frac{1}{2}\log (\zeta_2 t^{-1/2})+C
\end{equation*}
for any sufficiently large $t$. As a result,  the assertion \eqref{alphaclaim}
follows from the fact that
\begin{equation*}
\log (\zeta_2 t^{-1/2})=\log (-\zeta_{1}')-\frac{11}{2}\log t+O(1) \h \text{as $t\to\infty$} \quad\text{(see Lemma \ref{zetalem})}.
\end{equation*}
Thus, the conclusion follows from Lemma \ref{zetalem}.
\end{proof}

\bigskip
 \noindent
 {\bf Acknowledgements}\\
 The author was supported by JSPS KAKENHI Grant Number 26KJ0083.

\appendix
\section{A basic regularity lemma}
In this appendix, we introduce the following.
\begin{lem}
\label{apenlem}
Let $g(r)\in L^1(B_1)$.
Assume that $g_1:=r^2(-\log (r/e))^{-5/2}g(r)\in L^{\infty}(B_{1}^2)$ and $V(r)\in L^{\infty}(B_{1}^2)$. Let $w\in C^{1}_{0}(\overline{B_1})$  
be a radial solution of 
\[
-\Delta w= g(r)+V(r)w\quad \text{in $B_{1}^{2}$}.
\]
Then, 
\[
\lVert w\rVert_{L^{\infty}(B_{1}^2)}\le C(\lVert V\rVert_{L^{\infty}(B_{1}^2)}\lVert w\rVert_{L^2(B_{1}^2)}+\lVert g_1\rVert_{L^{\infty}(B_{1}^2)}),
\]
where $C>0$ is universal.
\end{lem}

\begin{proof}
We define $g_1=r^2(-\log (er))^{-5/2}g$. Then, 
by the representation formula, we obtain the following
$$
w(r)=\int_{r}^{1} \,ds \int_{0}^{s}(\frac{l}{s})(g+V(x)w) \,dl. 
$$
By the H\"older inequality, we obtain
$$
\int_{0}^{s}|(\frac{l}{s})(g+V(x)w) |\,dl\le C(\lVert g_1\rVert_{L^{\infty}} s^{-1}(-\log (s/e))^{-3/2}+\lVert V(r)\rVert_{L^{\infty}}\lVert w\rVert_{L^2}).
$$
As a result, we have
\begin{align*}
|w(r)|&\le C(\lVert V\rVert_{L^{\infty}}\lVert w\rVert_{L^2}+\lVert g_1\rVert_{L^{\infty}} \int_{r}^1 s^{-1}(-\log (s/e))^{-3/2}\,ds)\\
&\le C(\lVert V\rVert_{L^{\infty}}\lVert w\rVert_{L^2}+\lVert g_1\rVert_{L^{\infty}}).
\end{align*}
Therefore, we obtain the result.
\end{proof}


{\small
\bibliographystyle{abbrv}
\bibliography{Grow_up_rate}
}

\end{document}